\documentclass{amsart}
\usepackage[body={6in,9in}]{geometry}
\usepackage{amsrefs}
\usepackage{graphics}
\usepackage{amsmath}
\usepackage[all]{xy}

\newcommand{\sel}{\mathrm{Sel}^{(2)}}
\newcommand{\selfake}{\mathrm{Sel}_\mathrm{fake}^{(2)}}
\newcommand{\uu}{\underline{u}}
\newcommand{\FF}{\mathbb{F}}
\newcommand{\PP}{\mathbb{P}}
\newcommand{\QQ}{\mathbb{Q}}
\newcommand{\RR}{\mathbb{R}}
\newcommand{\ZZ}{\mathbb{Z}}
\newcommand{\Aut}{\mathrm{Aut}}
\newcommand{\Gal}{\mathrm{Gal}}
\newcommand{\Cov}{\mathrm{Cov}}
\newcommand{\Jac}{\mathrm{Jac}}

\newcommand{\ord}{\mathrm{ord}}
\newcommand{\disc}{\mathrm{disc}}
\newcommand{\genus}{\mathrm{genus}}
\newcommand{\rank}{\mathrm{rank}}

\newcommand{\kbar}{\overline{k}}

\newcommand{\fp}{\mathfrak{p}}
\newcommand{\fq}{\mathfrak{q}}

\newcommand{\sO}{\mathcal{O}}

\newcommand{\Cif}{\textbf{\textsf{if}}\ }
\newcommand{\Celse}{\textbf{\textsf{else}}\ }
\newcommand{\Cfor}{\textbf{\textsf{for}}\ }
\newcommand{\Creturn}{\textbf{\textsf{return}}\ }
\newcommand{\Cor}{\textbf{\textsf{or}}\ }
\newcommand{\Cand}{\textbf{\textsf{and}}\ }
\newcommand{\Cdef}{\textbf{\textsf{define}}\ }
\newcommand{\Ccall}{\textbf{\textsf{call}}\ }

\DeclareFontFamily{U}{wncy}{}
\DeclareFontShape{U}{wncy}{m}{n}{<->wncyr10}{}
\DeclareSymbolFont{mcy}{U}{wncy}{m}{n}
\DeclareMathSymbol{\Sha}{\mathord}{mcy}{"58}

\newtheorem{theorem}{Theorem}[section]
\newtheorem{lemma}[theorem]{Lemma}

\newtheorem{prop}[theorem]{Proposition}
\theoremstyle{definition}
\newtheorem{definition}[theorem]{Definition}
\newtheorem{question}[theorem]{Question}
\newtheorem{remark}[theorem]{Remark}
\newtheorem{example}[theorem]{Example}

\newcounter{nootje}
\renewcommand\check[1]
  {[*\thenootje]\marginpar{\tiny\begin{minipage}{20mm}\begin{flushleft}\thenootje : #1\end{flushleft}\end{minipage}}\addtocounter{nootje}{1}}
\begin{document}
\bibliographystyle{plain}
\title{Two-cover descent on hyperelliptic curves}
\parskip=0.1pt
\author{Nils Bruin}
\address{Department of Mathematics,
         Simon Fraser University,
         Burnaby, BC,
         Canada V5A 1S6}
\email{nbruin@sfu.ca}
\thanks{Research of the first author supported by NSERC}
 
\author{Michael Stoll}
\address{Mathematisches Institut,
         Universit\"at Bayreuth,
         95440 Bayreuth, Germany}
\email{Michael.Stoll@uni-bayreuth.de}


\date{\today}

\begin{abstract} We describe an algorithm that determines a set of
unramified covers of a given hyperelliptic curve, with the property that any
rational point will lift to one of the covers. In particular, if the algorithm
returns an empty set, then the hyperelliptic curve has no rational points. This
provides a relatively efficiently tested criterion for solvability of
hyperelliptic curves. We also discuss applications of this algorithm to curves
of genus $1$ and to curves with rational points.
\end{abstract}

\maketitle


\section{Introduction}

In this paper we consider the problem of deciding whether an algebraic curve 
$C$
over a number field $k$ has any $k$-rational points. We assume that $C$ is
complete and non-singular.
A necessary condition for $C(k)$ to be non-empty is that $C$ has a rational 
point over every extension of $k$.
In particular, for any place $v$ of $k$, the curve $C$ should have a rational 
point
over the completion $k_v$ of $k$ at $v$.

For curves of genus $0$, this is sufficient as well: if a genus $0$ curve $C$ 
has a
rational point over every completion $k_v$ of $k$, then $C(k)$ is non-empty.
A $k_v$-point of $C$ is referred to as a \emph{local} point of $C$ at $v$
and a $k$-point of $C$ is called a \emph{global} point. For a genus $0$ curve
$C$,
having a local point everywhere (at all places of $k$) implies having a
global point. Genus $0$ curves are said to obey the
\emph{local-to-global} principle for points.

The local-to-global principle is important from a computational point of view.
One can decide in finite time whether a curve has points everywhere locally:
for any curve $C$ over a number field $k$, the set $C(k_v)$ is nonempty for 
all
places $v$ outside an explicitly determinable finite set $S$. For the 
remaining
places $v\in S$, one can decide in finite time if $C(k_v)$ is non-empty as
well. See \cite{bruin:ternary} for some algorithms, in particular a quite
efficient algorithm for hyperelliptic curves. Thus, whether a curve has points
everywhere locally can actually be decided in finite time.

It is well known that curves of genus greater than $0$ do not always obey the
local-to-global principle. Most proofs of this phenomenon are based on the 
fact that
curves of positive genus can have unramified Galois-covers. 
In fact, if a curve~$C$ of positive genus has a
rational point~$P$, then the Abel-Jacobi map allows us to consider $C$ as a
non-singular subvariety of its Jacobian variety~$\Jac(C)$. Since the map
$\pi : \Jac(C)\to\Jac(C), Q \mapsto nQ + P$ is unramified, the pull-back 
$\pi^*(C)$ 
yields an unramified cover of $C$ that has a rational point mapping to~$P$.
More generally, an {\em $n$-cover} of~$C$ is a cover that is isomorphic to
one of this form over an algebraic closure $\kbar$ of~$k$.
Thus, if one can show that an algebraic curve $C$ does
not have $n$-covers that have a rational point, 
then it follows that $C$ has no rational points. Even though $C$ may have
points everywhere locally, it is possible that on every cover, rational points
can be ruled out by local conditions.

The major content of this paper is inspired by the following well-known
observation. Let $\phi:D\to C$ be an unramified cover of a curve $C$ over
a number field $k$ that is Galois over an algebraic closure $\kbar$ 
of~$k$.
It is a standard fact, going back to Chevalley and Weil \cite{chevweil:covers} 
 that there is a finite
collection of twists $\phi_\delta:D_\delta\to C$ of the given cover $\phi$ 
such
that any rational point on $C$ has a rational pre-image on one of the covers
$D_\delta$. We call such a set of covers a \emph{covering collection}.
Furthermore, at least in principle, such a finite collection of
covers is explicitly computable given a cover $\phi:D\to C$.

Thus, one approach for testing solvability of a curve $C$ that has points 
everywhere
locally is:
\begin{enumerate}
\item Fix $n \ge 2$.
\item Construct an $n$-cover $D$ of $C$. 
If no such covers exist, then $C(k)$ is empty.
\item Determine a covering collection associated to $D\to C$.
\item Test each member of the covering collection for local solvability. If
none of the members has points everywhere locally then no curve in the 
covering collection
has any rational points, and $C$ has no rational points either.
\end{enumerate}
Each of the covers might have a local
obstruction to having rational points, while the underlying curve $C$ has 
none.
Thus the procedure sketched above can actually prove that a curve does not 
obey
the local-to-global principle for points. 
See \cite{stoll:covers} for a detailed discussion, including the 
theoretical background and some links to
the Brauer-Manin obstruction against rational points.

In the present article we discuss a relatively efficient way of carrying out 
the
procedure sketched above for hyperelliptic curves. We consider unramified
covers $D/C$ over~$k$
such that for an algebraic closure $\kbar$ of $k$ we have that
$\Aut_{\kbar}(D/C)\simeq\Jac(C)[2](\kbar)$ as $\Gal(\kbar/k)$-modules. These
are exactly the \emph{two-covers} of $C$ over $k$. We write $\Cov^{(2)}(C/k)$ 
for
the set of isomorphism classes of two-covers of $C$ over $k$.

Our claim to efficiency stems from the fact that we avoid explicitly
constructing a covering collection. In fact, the method can be described 
without
any reference to unramified covers, see Section~\ref{sec:definition}.
Instead, we determine an abstract object that
(almost) classifies the isomorphism classes of two-covers. 
See Section~\ref{sec:geom} for how to construct the covers explicitly
from the information provided by the algorithm.

We observe that for
any field extension $L/k$ there is a well-defined map $C(L)\to\Cov^{(2)}(C/L)$
by sending a point to the two-cover that has an $L$-rational point above it.
For completions $k_v/k$, we show that this map is locally constant:
If two points $P_1,P_2\in C(k_v)$ lie sufficiently close, then they lift to 
the
same two-cover. This allows us to explicitly compute the $k_v$-isomorphism
classes of two-covers that have points locally at $v$. We can then piece
together this information to obtain the global isomorphism classes of 
two-covers
that have points everywhere locally.

We define $\sel(C/k)\subset\Cov^{(2)}(C/k)$ to be the set of everywhere 
locally
solvable two-covers of $C$. The algorithm computes a related object,
that we denote by $\selfake(C/k)$, which is a quotient of
$\sel(C/k)$. It classifies everywhere locally solvable two-covers, up to an
easily understood equivalence.
Since any curve $C$ with a rational point 
admits a globally and hence everywhere locally solvable two-cover,
$\selfake(C/k)=\emptyset$ implies that $C(k)$ is empty. 

A priori, the elements of $\selfake(C/k)$ represent two-covers that have
a model of a certain form (described in Section~\ref{sec:geom}), but
we will prove that every two-cover that has points everywhere locally does
have such a model, see Theorem~\ref{thm:twocovers} below.

In Section~\ref{sec:noratpoints} we illustrate how the algorithm presented in
Sections~\ref{sec:localmu}, \ref{sec:real}, and
\ref{sec:selfake} can be used to
show that a hyperelliptic curve has no rational points. This method was also 
used in a
large scale project \cite{brusto:smallgenus2} to determine the solvability of 
all genus $2$ curves with a
model of the form
$$y^2=f_6x^6+\cdots+f_0,\mbox{\quad where }f_i\in\{-3,-2,-1,0,1,2,3\}\mbox{ 
for
}i=0,\ldots,6.$$
In Section~\ref{sec:statistics} we describe some statistics that illustrate 
how
frequently one would expect that $\selfake(C/k)=\emptyset$ 
for curves of genus $2$.

Section~\ref{sec:ratpoints} shows how information obtained on $\selfake(C/k)$
can be used for determining the rational points on curves if $C(k)$ is
non-empty. This complements Chabauty-methods as described in
\cites{bruintract, bruin:ellchab, flynnweth:biell}. In earlier articles, the
selection of the covers requiring further attention was done by ad-hoc 
methods.
Here we describe a systematic and relatively efficient approach.

In Section~\ref{sec:genus1} we describe the results of applying the method to
curves of genus $1$. We find (unsurprisingly) that we recover well-known
algorithms for performing $2$-descents and second $2$-descents on elliptic
curves \cites{cassels:lecec, mss:seconddesc}. A practical benefit of this
observation is that to our knowledge, nobody has bothered to implement second
two-descent over arbitrary number fields, whereas our implementation in MAGMA 
\cite{magma} (which is available for download at 
\cite{brusto:implementation})
can immediately be used. We give an example by exhibiting
an elliptic curve over $\QQ(\sqrt{2})$ with a non-trivial Tate-Shafarevich
group.


\section{Definition of the fake $2$-Selmer group}\label{sec:definition}

Let $k$ be a field of characteristic $0$ and let
$C$ be a non-singular projective hyperelliptic curve of genus $g$ over $k$,
given by the affine model
$$C: y^2=f_{n}x^{n}+\cdots+f_0=f(x),\mbox{\quad where $f$ is square-free}.$$
If $n$ can be chosen to be odd, then $C$ has a rational
Weierstrass point. This is a special situation and, by not placing any
Weierstrass point over $x=\infty$, we see that we lose no generality by assuming
that $n$ is even, in which case $n=2g+2$. In practice, however, computations
become considerably easier by taking $n$ odd if possible. The construction below
can be adapted to accommodate for odd $n$. See
Remark~\ref{rmk:odddegree} below.

From here on we assume that $n$ is even unless explicitly stated otherwise. We consider the algebra
$$A=k[x]/(f(x))$$
and we write $\theta$ for the image of $x$ in $A$, so that $f(\theta)=0$. We
consider the subset of $A^*$ modulo squares and scalars (elements of $k^*$)
with representatives in $A^*$ that have a norm in $k^*$ that is equal to $f_n$
modulo squares:
$$H_k=\{\delta\in A^*/A^{*2}k^*: N_{A/k}(\delta)\in f_n k^{*2}\}.$$
The set $H_k$ might be empty (but see Question~\ref{Q}). As we will see, this implies that $C$ has no
rational points. This follows from the fact that we can define a map $C(k)\to
H_k$. First, we define the partial map:
$$\begin{array}{cccc}
\mu_k:&C(k)&\to&H_k\\
&(x,y)&\mapsto&x-\theta \, .
\end{array}$$
Here and in the following, we write $x-\theta$ instead of the correct, but
pedantic, \hbox{$(x-\theta) A^{*2} k^*$;} we hope that no confusion will result.
This definition of $\mu$ does not provide a valid image for any point
$(x_1,0)\in C(k)$. For any such point, we can write $f(x)=(x-x_1)\tilde{f}(x)$
and we define:
$$\mu_k( (x_1,0))=(x_1-\theta)+\tilde{f}(\theta).$$
Furthermore, if $f_n\in k^{*2}$ then the desingularisation of $C$ has two
points, say $\infty^+,\infty^-$ above $x=\infty$. We define
$$\mu(\infty^+)=\mu(\infty^-)=f_n,$$
where $f_n\equiv 1$ modulo $k^{*2}$.

\begin{remark}\label{rmk:odddegree}
While we lose no generality by assuming that $f(x)$ is of even degree, for
computational purposes it is often preferable to work with odd degree $f(x)$ as
well. We can use the definitions above if we replace the definition of $H_k$ by
$$H_k=\{\delta\in A^*/A^{*2}: N_{A/k}(\delta)\in f_n k^{*2}\}.$$
In this case, there is a unique point $\infty$ above $x=\infty$. We define
$$\mu(\infty)=f_n$$
\end{remark}

If $K$ is a field containing $k$ (we will consider a number field $k$ with a
completion $K$), the natural map $A\to A\otimes_k K$ induces the commutative
diagram
$$\xymatrix{
C(k)\ar[r]^{\mu_k}\ar[d]&H_k\ar[d]^{\rho_K}\\
C(K)\ar[r]^{\mu_K}&H_K}$$
If $k$ is a number field and $v$ is a place of $k$, we write $\mu_{k_v}=\mu_v$,
$\mu_k=\mu$ and $\rho_{k_v}=\rho_v$. For a number field $k$, we define
$$\selfake(C/k)=\left\{\delta\in H_k: \rho_v(\delta)\in\mu_v(C(k_v))\text{ for
all places }v\text{ of }k\right\} \subset H_k.$$
It is then clear that $\mu_k(C(k)) \subset \selfake(C/k)$.
In particular, $\selfake(C/k) = \emptyset$ implies that $C$ does not have
$k$-rational points.


\section{Geometric interpretation of $H_k$}\label{sec:geom}

In this section we give a geometric and Galois-cohomological interpretation of
the set $H_k$ and the map $\mu_k$ we defined in Section~\ref{sec:definition}.
The material in this section is not essential for the other sections in this
text. 
\begin{definition} Let $C$ be a non-singular curve of genus $g$ over a field $k$. A
non-singular
absolutely irreducible cover $D$
of $C$ is called a \emph{two-cover} if $D/C$ is unramified and Galois over a separable closure
$\kbar$ of $k$ and $\Aut_{\kbar}(D/C)\simeq (\ZZ/2\ZZ)^{2g}$. We denote the set
of isomorphism classes of $2$-covers of $C$ over $k$ by
$\Cov^{(2)}(C/k)$.
\end{definition} 

This definition is motivated by the fact that if $C$ can be embedded in
$\Jac(C)$, then such a cover can be constructed by taking $D$ to be the
pull-back of $C$ along the multiplication-by-two morphism
$[2]:\Jac(C)\to\Jac(C)$. Furthermore, over $\kbar$, all two-covers of $C$ are
isomorphic, and $\Aut_{\kbar}(D/C)\simeq \Jac(C)[2](\kbar)$ as
$\Gal(\kbar/k)$-modules in a canonical way.

Let $C$ be a curve and suppose that $D_1,D_2$ are two-covers of $C$ over $k$. Let $\phi: D_1\to D_2$
be an isomorphism of $C$-covers over $\kbar$. We can associate a $\Gal(\kbar/k)$-cocycle to this via
$$\begin{array}{rccccl}
\xi:&\Gal(\kbar/k)&\to& \Aut_{\kbar}(D_2/C)&=&\Jac(C)[2](\kbar)\\
&\sigma&\mapsto&\phi^\sigma\circ\phi^{-1}
\end{array}$$
The cocycle $\xi$ is trivial in $H^1(k,\Jac(C)[2])$ precisely if $D_1$ and
$D_2$ are isomorphic as $C$-covers over $k$. Furthermore, given a cocycle $\xi$,
one can produce a twist of $D_\xi$ of a given cover $D$:

\begin{theorem}[\cite{sil:AEC}*{Chapter X, Theorem 2.2}]
\label{thrm:twist}
If $(D \to C) \in \Cov^{(2)}(C/k)$, then the above construction gives
a bijection $\Cov^{(2)}(C/k) \to H^1(k,\Jac(C)[2])$.
\end{theorem}

We now interpret the set $H_k$ defined in Section~\ref{sec:definition} in terms of two-covers.
Using the notation from the previous section, consider $\delta\in A^*$. There
are unique quadratic forms
$$Q_{\delta,i}(\uu)\in k[u_0,\ldots,u_{n-1}]$$
such that the identity below holds in $A[u_0,\dots,u_{n-1}]$:
$$\delta(u_0+u_1\theta +\cdots +
u_{n-1}\theta^{n-1})^2=\sum_{i=0}^{n-1}Q_{\delta,i}(\uu)\theta^i.$$
We consider the projective variety over $k$ in $\PP^{n-1}$ described by
$$D_\delta:Q_{\delta,2}(\uu)=\cdots=Q_{\delta,n-1}(\uu)=0$$
This curve $D_\delta$  is a degree $2^{n-1}$ cover of $\PP^1$ via the function
$$x(\uu)=-\frac{Q_{\delta,0}(\uu)}{Q_{\delta,1}(\uu)}.$$
Furthermore, if $f_nN_{A/k}(\delta)=v^2$ for some $v\in k$,
we can define a function
$$y(\uu)=v\,
   \frac{N_{A[\uu]/k[\uu]}\left(\sum\limits_{i=0}^{n-1}u_i\theta^i\right)}{(Q_{\delta,1}(\uu))^{n/2}}$$
which gives us morphisms $\phi_{\delta,\pm}$, depending on the choice of $v$, represented by
$$\begin{array}{cccc}
\phi_{\delta,\pm}:& D_\delta&\to&C\\
&(u_0:\cdots:u_{n-1})&\mapsto&(x(\uu),\pm y(\uu))
\end{array}$$
It is proved in \cites{bruintract,brufly:towtwocovs} that the cover $D_\delta/C$ is
a two-cover. Furthermore, if $\delta_1,\delta_2$ represent distinct classes in $H_k$, the covers
$D_{\delta_1}$ and $D_{\delta_2}$ are not isomorphic over $k$. 

For a fixed $\delta$,
the two morphisms $\phi_{\delta,+}$ and $\phi_{\delta,-}$ show that $D_\delta$
is a two-cover of $C$ in two ways.
These are related via the hyperelliptic involution on $C$, denoted by
$$\begin{array}{cccc}
\iota:&C&\to& C\\
&(x,y)&\mapsto&(x,-y)\,.
\end{array}$$
With this notation, we have $\phi_{\delta,-}=\iota\circ\phi_{\delta,+}$.

The hyperelliptic involution
induces a map $\iota^*:\Cov^{(2)}(C/k)\to \Cov^{(2)}(C/k)$ via $\iota^*(\phi)=\iota\circ\phi$.
Since elements of $H_k$ define two-covers up to $\iota^*$ we have:
$$H_k\subset \Cov^{(2)}(C/k)/\langle \iota^*\rangle.$$
Whether $\iota^*$ is the identity depends on whether there exists an isomorphism $D\to D$ over $k$ such
that the diagram below commutes.
$$\xymatrix{
D\ar[dr]_{\phi}\ar@{-->}@/^2ex/^?[rr]&&D\ar[dl]^{\iota\circ\phi}\\
&C}$$
We do not need it here, but we quote a result from \cite{poosch:cyclic}
(Thm.~11.3)
that characterizes whether $\iota^*$ acts trivially on $H^1(k,\Jac(C)[2])$.
\begin{prop} Let $C$ be as defined above. Then $\iota^*$ acts trivially on
$H^1(k,\Jac(C)[2])$ if and only if
$C$ has an odd degree Galois invariant and $\iota$-symmetric divisor class. 
A curve $C: y^2=f(x)$, where $f$ is square-free of degree $n$, has such a
divisor class precisely
\begin{itemize}
\item when $n$ is odd, or
\item when $n\equiv 0\pmod{4}$ and $f$ has an odd degree factor, or
\item when $n\equiv 2\pmod{4}$ and $f$ has an odd degree factor or is the product
of two quadratically conjugate factors.
\end{itemize}
\end{prop}

If $k$ is a number field, we define the \emph{2-Selmer set} of $C/k$ to be the set of everywhere
locally solvable $2$-covers:
$$\sel(C/k)=\left\{ (\phi:D\to C) \in \Cov^{(2)}(C/k) :
D(k_v)\neq\emptyset \text{ for all places $v$ of $k$}.\right\}$$
We define the \emph{fake} $2$-Selmer set to be the everywhere locally
solvable $2$-covers of the form $D_\delta$:
$$\selfake(C/k):=H_k\cap\sel(C/k)/\langle \iota^*\rangle.$$
This is consistent with the definition given in Section~\ref{sec:definition}:
if $P\in C(k)$ and $\delta=\mu(P)$, then $D_\delta$ has a rational point that
maps to $P$. Therefore, if $\delta\in H_k$ restricts to an element in
$\mu_{k_v}(C(k_v))$ for all places $v$ of $k$, then $D_\delta$ has a $k_v$-rational
point for each $v$.

This argument also shows that if $H_k$ is empty then $C(k)$ is empty. 
More precisely, the set $H_k$ classifies two-covers of the form $D_\delta$,
so if it is empty, this represents an obstruction against the existence
of a two-cover of this specific form. We will now show that every two-cover
$D \to C$ such that $D$ has points everywhere locally can be realized
as a cover~$D_\delta$.

\begin{theorem} \label{thm:twocovers}
Let $\phi : D \to C$ be a two-cover such that $D$ has points everywhere
locally. Then there is $\delta \in \selfake(C/k)$ such that $\phi$
is isomorphic to $\phi_{\delta,+}$ or~$\phi_{\delta,-}$.
In particular, 
$$\selfake(C/k) = \sel(C/k)/\langle \iota^*\rangle\,.$$
\end{theorem}

\begin{proof}
We first note that all divisors $\phi^*(P)$ on~$D$, where $P$ is a
Weierstrass point on~$C$, are linearly equivalent. This is a geometric
statement, so we can assume $k$ to be algebraically closed; then 
$\phi \simeq \phi_{1,+}$, and on~$D_1$, it is easy to see that the divisors
in question all are hyperplane sections (for example, if $P = (\theta, 0)$,
then $\phi^*(P)$ is given by the vanishing of 
$u_0 + \theta u_1 + \dots + \theta^{n-1} u_{n-1}$). Let us denote the
class of all these divisors by~$W$. Since the Galois action maps
Weierstrass points to Weierstrass points, $W$ is defined over~$k$.

The assumption that $D$ has points everywhere locally implies that
every $k$-rational divisor class contains a $k$-rational divisor.
So $W$ induces a projective embedding $D \to \PP^N$ such that the
pull-backs of Weierstrass points on~$C$ are hyperplane sections.
Now consider the function $x - \theta \in A(C)$. Its  
pull-back to~$D$ has divisor $2 \phi^*((\theta,0)) - V$, where
$V$ is a $k$-rational effective divisor whose class is twice that of
a hyperplane section. In terms of the coordinates on~$\PP^N$ (projective $N$-space over the \'etale algebra $A$), this means
that we can write
\[ (x - \theta) \circ \phi = \delta \frac{\ell^2}{q} \]
with a constant $\delta \in A^*$, a linear form~$\ell$ with
coefficients in~$A$, and a quadratic form~$q$ with coefficients
in~$k$. Taking norms (recall that $f_n N_{A/k}(x - \theta) = y^2$), 
we find that
\[ (y \circ \phi)^2
     = f_n N_{A/k}(\delta) \left(\frac{N_{A/k}(\ell)}{q^{n/2}}\right)^2 ,
\]
so that $\delta$ represents an element of~$H_k$.

We can write the linear form~$\ell$ as
\[ \ell = \ell_0 + \ell_1 \theta + \dots + \ell_{n-1} \theta^{n-1} \]
where the $\ell_i$ are linear forms with coefficients in~$k$. 
We obtain a map $D \to D_\delta \subset \PP^{n-1}$, given by 
$(\ell_0 : \dots : \ell_{n-1})$, which is the desired isomorphism.
\end{proof}


\section{Computing the local image of $\mu$ at non-archimedean places}
\label{sec:localmu}

In this section, we assume that $k$ is a non-archimedian complete local field of characteristic $0$. We
write $\sO$ for the ring of integers inside $k$, $\ord$ for the discrete valuation and
$|.|$ for the absolute value on $k$. Furthermore $\pi$ will be a uniformizer. We will
assume that $\sO/(\pi)$ is finite of characteristic $p$ and that $R\subset \sO$ is a complete set of
representatives for $\sO/(\pi)$.

Let $f(x)\in\sO[x]$ be a square-free polynomial and suppose that
$f(x)=g_1(x)\cdots g_m(x)$ is a factorisation into irreducible polynomials 
with $g_i \in \sO[x]$. 
If we write $L_i=k(\theta_i)=k[x]/(g_i(x))$ then
$A\simeq L_1\oplus\cdots\oplus L_m$ and
$$A^*/A^{*2}\simeq(L_1^*/L_1^{*2})\times\cdots\times (L_m^*/L_m^{*2}).$$

The following definitions and lemma allow us to find $\mu(C(k))$ without
computation in many cases.

\begin{definition}
Let $L$ be a local field and let $\delta \in L^*$. We say the class of $\delta$
in $L^*/L^{*2}$ is \emph{unramified} if the extension $L(\sqrt{\delta})/L$ is
unramified. If the residue characteristic
$p$ is odd, this just means that $\ord_L(\delta)$ is even.
\end{definition}

\begin{definition}\label{def:unram}
Let $A$ be an \'etale algebra over a local field $k$ and suppose that
$A\simeq L_1\oplus\cdots\oplus L_m$ is a decomposition of $L$ into irreducible
algebras. Then we say that $\delta\in A^*/A^{*2}$ is \emph{unramified} if the
image of $\delta$ in each of $L_i^*/L_i^{*2}$ is unramified.

We say that an element of $A^*/A^{*2} k^*$ is \emph{unramified} if it can
be represented by an unramified $\delta \in A^*/A^{*2}$.
\end{definition}

\begin{lemma}\label{lemma:unram}
Suppose that $f\in\sO[x]$, that the residue characteristic $p$ is odd, that 
$\ord_k(\disc(f))\leq 1$ and that the
leading coefficient of $f$ is a unit in $\sO$. Then $\mu(C(k))$ consists
of unramified elements.

Furthermore, if in addition $q=\#\sO/(\pi)$ satisfies
$$\sqrt{q}+\frac{1}{\sqrt{q}} > 2(2^{2g}(g-1)+1)$$
then $\mu(C(k))$ consists exactly of the unramified elements
of $H_k \subset A^*/A^{*2} k^*$.
\end{lemma}
(Compare Prop.~5.10 in~\cite{stoll:descent} for a similar statement in the
context of 2-descent on the Jacobian of~$C$.)

\begin{proof}
Let $L/k$ be a splitting field of $f$ and let $\theta_1,\ldots,\theta_n$ be the
roots of $f$. Since the leading coefficient of $f$ is a unit, we have
$\theta_i\in\sO_L$.  We extend $\ord = \ord_k$ to $L$ by writing
$\ord(y)=\frac{1}{e(L/k)}\ord_L(y)$, where $e(L/k)$ is the ramification index of
$L/k$.

First we prove that $e(L/k)\leq 2$. Note that, since $\ord(\disc(f))\leq 1$, we
have $e(k[\theta_i]/k)\leq 2$. If $e(k[\theta_i]/k)=1$ for all $i$, then
$e(L/k)=1$. Therefore, suppose that $e(k[\theta_1]/k)=2$. Write
$f(x)=(x-\theta_1)\tilde{f}(x)$. Then
$$\disc(f)=\disc(\tilde{f}) (\tilde{f}(\theta_1))^2$$
If $e(k[\theta_1]/k)=2$ then $\ord(\tilde{f}(\theta_1))> 0$, so in fact
$\ord(\tilde{f}(\theta_1))\geq \frac{1}{2}$. But then, from
$$1=\ord(\disc(f))=\ord(\disc(\tilde{f}))+2\,\ord(\tilde{f}(\theta))$$
it follows that $\ord(\disc(\tilde{f}))=0$ and hence that $L/k[\theta_i]$ is
unramified for $i \ge 2$, so $e(L/k)\leq 2$ and $\ord$ takes values in 
$\frac{1}{2}\ZZ$ on~$L^*$.

This allows us to conclude that the roots of $f(x)$ are $p$-adically widely
spaced. From
$$1\geq\ord(\disc(f))=\ord\Bigl(\prod_{i<j}
(\theta_i-\theta_j)^2\Bigr)=2\sum_{i<j}\ord(\theta_i-\theta_j)$$
it follows that $\ord(\theta_i-\theta_j)=0$ for all but at most one pair
$\{i,j\}$, say $\{1,2\}$. If $L/k$ is ramified then we have
$\ord(\theta_1-\theta_2)=\frac{1}{2}$ and $(x-\theta_1)(x-\theta_2)$ is an
irreducible quadratic factor of $f(x)$ over $k$.

We now consider a point $(x,y)\in C(k)$ with $x\in \sO$. Since $f(x)$ is a
square in $k$ and the leading coefficient is a unit, we have that
$$2\mid \sum_{i=1}^n \ord(x-\theta_i)$$
Note, however, that $\ord(\theta_i-\theta_j)\geq \min
(\ord(x-\theta_i),\ord(x-\theta_j))$. A priori, we could still have
$\ord(x-\theta_1)=\ord(x-\theta_2)=\frac{1}{2}$ (note that these orders must be
equal because $x-\theta_1$ and $x-\theta_2$ are Galois-conjugate)
and $\ord(x-\theta_i)=0$ for
$i=3,\ldots,n$, but then $f(x)$ has odd valuation and hence is not a square in
$k$. It follows that $\ord(x-\theta_i)=0$ holds
for all but at most one $i$ and therefore that all of $\ord(x-\theta_i)$ are even.
This proves that $\mu(x,y)$ is unramified for any point $(x,y)\in C(k)$ with
$x\in\sO$.

If $x\notin \sO$, then $\ord(x) < 0 \le \ord(\theta_i)$ for all~$i$, so
(since $p$ is odd) $(x-\theta_i)/x$ is a square in~$k(\theta_i)$.
So $\mu(x,y)$ is in $A^{*2} k^*$, i.e., trivial.
It follows that the image $\mu(C(k)) \subset H_k$ is unramified.

Conversely, if $\ord(\disc(f))=0$ and $\delta \in H_k$ is unramified, then $D_\delta$
as defined in Section~\ref{sec:geom} can be presented by a model with good
reduction (by taking $\delta$ to be a unit). 
Since $D_\delta$ is an unramified cover of degree $2^{2g}$ over a genus
$g$ curve $C$, we can compute using the Riemann-Hurwitz formula that
$$\genus(D_\delta) = 2^{n-3}(n-4) + 1 = 2^{2g}(g-1)+1.$$
The Weil bounds for the number of points on a nonsingular curve over a finite
field of cardinality $q$ imply that if $q$ satisfies the inequality stated in
the lemma, then the reduction of $D_\delta$ has a nonsingular point.
Hensel's lifting theorem tells us that $D_\delta(k)$ is non-empty and
therefore that $\delta\in \mu(C(k))$.

If $\ord(\disc(f))=1$ and $\delta \in H_k$ is unramified, we claim that the reduction $\overline{D}_\delta$ of $D_\delta$ is a singular curve of genus
$2^{2g-2}(2g-3)+1$, with a unique singularity at which $2^{2g-1}$ branches meet.
If the desingularization of~$\overline{D}_\delta$ has more than $2^{2g-1}$ points, then $\overline{D}_\delta$ must have a non-singular point, so via Hensel's lemma,
$D_\delta(k)$ is non-empty and $\delta\in\mu(C(k))$.

Again, from the Weil bounds it follows that this is the case if
\[ \sqrt{q} + \frac{1}{\sqrt{q}}
     > 2 \bigl(2^{2g-2}(2g-3) + 1\bigr) + \frac{2^{2g-1}}{\sqrt{q}} . \]
It is straightforward to check that for $g>0$, this is a weaker condition than the one stated in the lemma.

We now prove the claim. By taking $\delta$ to be a unit, we see that we
can construct $\overline{D}_\delta$ by applying the construction of~$D_\delta$
over the residue class field~$\FF = \sO/(\pi)$. The reduction of~$f$
has a unique double root in~$\FF$ and otherwise simple roots in an algebraic
closure of~$\FF$. We can assume the double root to be at~$x = 0$. Since the
statement is geometric, we assume that $\FF$ is algebraically closed.
Let $\theta_2, \dots, \theta_{n-1}$ be the simple roots. We obtain equations
defining $\overline{D}_\delta$ by eliminating $X$ and~$Z$ from the following
system
\begin{gather*}
  X = \delta_0 z_0^2, \qquad -Z = z_0(\delta_1 z_0 + 2 \delta_0 z_1) \\
  X - \theta_j Z = \delta_j z_j^2, \qquad j \in \{2, \dots, n-1\}.
\end{gather*}
Here the first pair of equations is obtained from
the component $\FF[x]/(x^2)$ of the algebra $\FF[x]/(f(x))$ in the following way:
if $t$ is the image of $x$ in $\FF[x]/(x^2)$, we write elements of this
algebra in the form $a_0 + a_1 t$. We get the first two equations by setting
\[ X - Z t = (\delta_0 + \delta_1 t) (z_0 + z_1 t)^2 \]
and comparing coefficients.

Substituting the expressions for $X$ and~$Z$ into the second set of equations,
we obtain
\[ z_0\bigl(\delta_0 z_0 + \theta_j(\delta_1 z_0 + 2 \delta_0 z_1)\bigr)
    = \delta_j z_j^2, \quad j \in \{2, \dots, n-1\} \,.
\]
It can be easily checked that the only singular point of this curve is where
all variables but~$z_1$ vanish. Projecting away from this point, we obtain
a smooth curve in~$\PP^{n-2}$ that is the complete intersection of $n-3 = 2g-1$
quadrics and therefore has genus~$2^{2g-2}(2g-3)+1$. Since (away from~$z_0 = 0$)
we can reconstruct~$z_1$ from the remaining coordinates, this projection
is a birational map, hence the (geometric) genus of~$\overline{D}_\delta$
is as given. The points on the smooth model that map to the singularity
on~$\overline{D}_\delta$ have $z_0 = 0$ (this is where the function $z_1/z_0$
is not defined on the smooth model), and it can be checked that there are
exactly $2^{2g-1} = 2^{n-3}$ such poins ($z_0 = 0$, and the ratios of the
squares of the other $n-2$ coordinates are fixed and nonzero). Hence the smooth
points of~$\overline{D}_\delta$ are in bijection with the remaining points
of the smooth model.

See also \cite{bruintract}*{Section~3.1} for a more in-depth discussion of this model of $D_\delta$ and \cite{brufly:towtwocovs} for a characterisation of $D_\delta$ as a maximal elementary $2$-cover of $\PP^1$, unramified outside $\{\theta_1,\ldots,\theta_n\}$.
\end{proof}

In other cases, when a prime divides the discriminant of $f$ more than once,
the residue field is too small or the leading coefficient of $f$ is not a unit
or $k$ has even residue characteristic, we have to do some computations to find the image of $\mu$. In
principle, one could construct $H_k$ as a finite set, enumerate all $D_\delta$
and test each of these for $k$-rational points. We present a more efficient
algorithm that instead enumerates points from $C(k)$ up to some sufficient
precision.

Computational models for complete local fields usually consist of computing in the finite
ring $\sO/\pi^e$ for some sufficiently large $e$, which is usually referred to as the
\emph{precision}. The following definition allows us to elegantly state
precision bounds. The variable $\epsilon$ is an indeterminate.
$$\begin{array}{cccc}
\ord:&\sO[\epsilon]&\to&\ZZ\\
&\displaystyle{\sum a_i\epsilon^i}&\mapsto&\displaystyle{\min_i \ord(a_i)}
\end{array}$$
It follows that
$$\ord\left(\sum a_ix^i\right)\geq\ord\left(\sum a_i\epsilon^i\right)
\text{ for all }x\in\sO,$$
and hence that if $f(x)\in\sO[x]$ and $v=\ord(f(x_1+\pi^e\epsilon)-f(x_1))$, then the
value of $f(x_1)$ is determined in $\sO/\pi^v$ by the value of $x_1$ in $\sO/\pi^e$.

\begin{lemma}
\label{lem:sqrdet}
Suppose $g(x)\in k[x]$ is an irreducible polynomial and that 
for some $e\in\ZZ_{\geq 0}$ and $x_0 \in \sO$, we have
$$\ord(g(x_0+\epsilon\pi^e)-g(x_0))>\ord(g(x_0))$$
Let $L=k[\theta]=k[x]/(g(x))$. Then for any $x_1\in x_0+\pi^{e+\ord(4)}\sO$ we have
$$(x_0-\theta)(x_1-\theta)\in L^{*2}$$
\end{lemma}

\begin{proof}
(Compare \cite{stoll:descent}, Lemma~6.3.)
Using that $g(x)=g_0N_{L/k}(x-\theta)$, where $g_0$ is the leading coefficient of $g(x)$,
we have
$$\ord\left(N_{L/k}\left(\frac{x_0+\epsilon\pi^e-\theta}{x_0-\theta}\right)-1\right)>0.$$
Writing $\ord_L$ for the valuation on $L$, we have that $\ord_L(\pi)$ is the ramification
index of $L/k$ and that
$$\ord_L\left(\left(\frac{x_0+\epsilon\pi^e-\theta}{x_0-\theta}\right)-1\right)>0.$$
With some elementary algebra we see that if $x_1\in x_0+ \pi^{e+\ord(4)}\sO$ then
$$\frac{x_1-\theta}{x_0-\theta}\in 1+ \pi^{\ord_k(4)+1}\sO_L\subset L^{*2}$$
and hence that $(x_0-\theta)(x_1-\theta)$ is a square in $L$.
\end{proof}

This lemma forms the basis for a recursive algorithm that determines the image of $\mu$ for
points $(x_1,y_1)\in C(k)$, with $x_1\in x_0+\pi^e\sO$. A similar procedure is
described in \cite{stoll:descent}*{page~270}. There are a few differences:
\begin{itemize}
\item we fully describe the algorithm for places with even residue
characteristic as well,
\item we do not place extra assumptions on the Newton polygon of $f(x)$,
\item the polynomial $f(x)$ does not change upon recursion. The algorithm in \cite{stoll:descent}
applies variable substitutions to $f(x)$. This will usually involve a lot of
arithmetic with the polynomial coefficients of $f$ to a relatively high $\pi$-adic
precision. We therefore expect that our algorithm will run slightly faster than
\cite{stoll:descent}, especially for small residue fields.
\end{itemize}

We first give an informal outline of the algorithm. We build up the possible
$x_1\in x_0+\pi^e\sO$, one $\pi$-adic digit at the time. At each stage, we make
sure that $f(x_1)$ is indistinguishable from a square (step 4 below).
After finitely many steps Lemma~\ref{lem:sqrdet} guarantees that the digits we
have fixed for $x_1$, determine the image of $x_1-\theta\in L^{*}/L^{*2}$. We then
add that value to the set $W$, unless $x_1$ lies close to the $x$-coordinate of
a Weierstrass point.

Hence, the purpose of the routine below is not to return a useful value, but to
modify a global list $W$ such that all values of $\mu( C(k)\cap x^{-1}(x_0+\pi^e\sO))$
outside those corresponding to Weierstrass points
are appended to $W$.

When first called, $G_0$ contains the irreducible factors of $f$. This set gets
adjusted upon recursion. 
The parameter $c_0$ is an auxiliary parameter that plays a role in keeping track
of whether the conditions of Lemma~\ref{lem:sqrdet} are met when $\ord(4)\neq 0$.
Its value is irrelevant if $G_0$ contains at least two polynomials or at least
one polynomial of degree larger than $1$.
Recall that $R$ is a complete set of representatives of $\sO/(\pi)$ in $\sO$.

\noindent\Cdef \textsf{SquareClasses}($x_0$, $e$, $G_0$, $c_0$):

\begin{enumerate}
\item[1.] \Cfor$r\in R$\textbf{\textsf{:}}
\item[2.] \hspace{1em} $x_1:= x_0+ \pi^e r$
\item[3.] \hspace{1em} $v_1=\ord(f(x_1))$;
                          $E_1:=\ord(f(x_1+\epsilon\pi^{e+1})-f(x_1))$

\item[4.] \hspace{1em} \Cif $E_1\leq v_1$ \Cor
($2\mid v_1$ \Cand $f(x_1)/\pi^{v_1}\in (\sO/\pi^{E_1-v_1})^{*2}$):
\item[5.] \hspace{2em}
  $G_1:=\{g\in G_0: \ord(g(x_1+\epsilon\pi^{e+1})-g(x_1))\leq\ord(g(x_1))\}$
\item[6.] \hspace{2em}
  \Cif $G_1=\emptyset$ \Cor ($G_1=\{g\}$ \Cand $\deg(g)=1$):
\item[7.] \hspace{3em}
  \Cif $G_0\neq G_1$\textbf{\textsf{:}} $c_1:=\ord(4)$ \Celse $c_1:=c_0-1$
\item[8.] \hspace{3em}
  \Cif $c_1 = 0$:
\item[9.] \hspace{4em}
  \Cif $G_1 = \emptyset$:
  Add the class of $\mu(x_1)$ to $W$.
\item[10.] \hspace{4em}
\Creturn
\item[11.] \hspace{2em}
  \Ccall \textsf{SquareClasses}($x_1$, $e+1$, $G_1$, $c_1$)

\end{enumerate}

\noindent\textbf{Explanation:}
\begin{enumerate}
\item[ad 1.] We split up $x_0+\pi^{e}\sO$ into smaller neighbourhoods
$x_1+\pi^{e+1}\sO$
\item[ad 3.] Here $E_1$ is the \emph{precision} to which $f(x_1)$ is determined:
$E_1$ is the largest integer such that
$f(x_1+\pi^{e+1}\sO)\subset f(x_1)+\pi^{E_1}\sO$.
\item[ad 4.] We only need to consider neighbourhoods that may contain a point
$(x_1,y_1)\in C(k)$. This is only the case if $f(x_1)$ is a square up to the
precision to which it is determined. The sets $(\sO/\pi^{E_1-v_1})^{*2}$ are
only needed for $1\leq E_1-v_1\leq \ord(4)+1$ and can be precomputed.
\item[ad 5.] The correctness of this algorithm hinges on Lemma~\ref{lem:sqrdet}.
We let $G_1$ be the subset of $G_0$ for which the lemma does not apply yet.
\item[ad 6.] Note that for any path in the recursion, in finitely many steps,
the value of any $g\in G_0$ on
$x_1+\pi^{e+1} \sO$ is determined up to a sufficiently high precision to be
distinguished from $0$, or $x_1$ is a good approximation to the root of exactly
one degree $1$ element of $G_0$.
\item[ad 7.] Informally, Lemma~\ref{lem:sqrdet} states that the value of
$g(x_1)$ in $L^*/L^{*2}$ is determined if at least $\ord(4)$ $\pi$-adic digits 
of $x_1$
\emph{beyond} the ones needed to distinguish $g(x_1)$ from $0$ are known. Since
every recursion in 11 has the effect of fixing another digit, we need a device
to count $\ord(4)$ more iterations. If $G_1$ is different from $G_0$, then we
have just gained a digit that helps to establish that $g(x_1)\neq 0$ for some
$G\in G_0$, and hence we should initialize $c_1=\ord(4)$, to count the full
$\ord(4)$ digits that still need to be added to $x_1$. Otherwise, we have just
determined one more step, so we should set $c_1=c_0-1$.
\item[ad 8.] If $c_1=0$ then the conditions of Lemma~\ref{lem:sqrdet} are satisfied
for all $g\notin G_1$: If $c_1=0$ and $g_i\notin G_1$ then
$(x_1-\theta_i)(x_2-\theta_i)$ is a square in $L_i^*$ for all $x_2\in
x_1+\pi^{e+1}\sO$.

Therefore, if $G_1=\emptyset$, then all $P\in
C(k)$ with $x(P)\in x_1+\pi^{e+1}\sO$ have the same image for $x(P)-\theta_i$ in
$L_i^*/L_i^{*2}$ and therefore, $\mu(P)=x_1-\theta$ in $A^*/A^{*2}$. In
addition, we know that
$$f(x_1)=f_0\prod_i N_{L_i/k} (x_1-\theta_i)$$
is a square due to the test in step 4. This verifies that such points $P$ do
exist and thus that $x_1-\theta$ represents an element of $\mu(C(k))$.

Alternatively, suppose that $G_1$ contains one polynomial, of degree $1$.
We write $G_1=\{g_j(x)\}$ with $g_j(x)=a(x-\theta_j)$. For any point $P\in C(k)$ with
$x(P)\in x_1+\pi^{e+1}\sO$, we have that $f(x(P))$ is a square. However, since
$$f(x)=f_0 (x-\theta_j)\prod_{i\neq j} N_{L_i[x]/k[x]} (x-\theta_i)$$
and $(x(P)-\theta_i)(x_1-\theta_i)$ is a square in $L_i$ for all $i\neq j$, we
see that the square class of $x(P)-\theta_j$, if non-zero, must be constant too
and that all such points $P$ have $\mu(P)=\mu((\theta_j,0))$. Therefore, if we
take care to record the images of all degree $1$ Weierstrass points of $C$
beforehand, these points are taken care of.

\item[ad 11.] If the test in step 6 does not hold true, or if $c_1\neq 0$ then
we cannot guarantee that $\mu$ is constant for all $P\in C(k)$ with $x(P)\in
x_1+\pi^{e+1}\sO$. In this case, we call the same routine again, to refine our
search. As remarked for step 6, the condition there will be satisfied after
finitely many recursion steps and then after at most $\ord(4)$ steps, we will
have $c_1=0$ as well.

\end{enumerate}

\noindent\Cdef \textsf{LocalImage}($f$):

\begin{enumerate}
 \item[1.] Let $g_1\cdot\cdots\cdot g_m=f$ be a factorization into irreducible
 polynomials.
 \item[2.] $A:=k[\theta]=k[x]/f(x)$; $H:=A^*/A^{*2}$
 \item[3.] $W:=\left\{(x_1-\theta)+\left.\frac{f(x)}{x-x_1}\right|_{x=x_1}\text{
 in }H: x_1 \text{ a root of $f(x)$ in }k\right\}$
 \item[4.] $\mu: x\mapsto x-\theta$ in $H$
 \item[5.] $G:=\{g_1,\ldots,g_m\}$
 \item[6.] \Ccall \textsf{SquareClasses}($0$, $0$, $G$, $-1$)
 \item[7.] \Cif $n$ is even: $\tilde{f}:= f_0 x^n+ \cdots +f_n$
 \Celse
 $\tilde{f}:= f_0 x^{n+1}+ \cdots +f_nx$
 \item[8.] Let $\tilde{G}$ consist of a factorisation of $\tilde{f}$ into
 irreducibles.
 \item[9.] $\tilde{\mu}: x\mapsto x(1-x\theta)$
 \item[10.] \Cif $n$ is odd \Cor $f_n$ is a square: Add $1$ to $W$
 \item[11.] \Cif $n$ is even \Cand $f_n$ is a square: Add $x$ to $\tilde{G}$
 \item[12.] \Ccall \textsf{SquareClasses}($0$, $1$, $\tilde{G}$, $-1$) while using
 $\tilde{f}$ and $\tilde{\mu}$ instead of $f$ and $\mu$.
 \item[13.] \Creturn $W$
\end{enumerate}

\noindent\textbf{Explanation:}
\begin{enumerate}
\item[ad 2.] The algorithm we describe determines $\mu(C(k))$ for
$\mu: C(K)\to A^*/A^{*2}$ at no extra cost. If $n$ is even,
we need to take the image
under the map $A^*/A^{*2}\to A^*/A^{*2}k^*$ in order to find a proper
interpretation of the computed set.
\item[ad 3.] We initialize $W$ with the images of the degree $1$ Weierstrass points
under $\mu$. Thus, when we call \textsf{SquareClasses} and find ourselves with $c_1=0$
and $G_1$ non-empty, then the possible image under $\mu$ has already been
accounted for.
\item[ad 4.] We initialize $\mu$ with the definition that works for most points,
for use in \textsf{SquareClasses}. (See step 10 for why we are explicit about this here.)
\item[ad 6.] We now call \textsf{SquareClasses} to add to $W$ the images $\mu(P)$ of points $P\in
C(k)$ with $x(P)\in\sO$. Given that $G$ consists of the full factorization of
$f$, the value of $c_0$ passed to \textsf{SquareClasses} is irrelevant. We pass
the dummy value of $-1$.
\item[ad 7.] Note that for the remaining points, we have $1/x(P)\in\pi\sO$.
Therefore, by making a change of variables $z=1/x$ and $w=y/x^{\lceil n/2\rceil}$, we are
left with finding the images of the points on
$$w^2=f_0 z^n+\cdots+f_n \text{ if $n$ is even or }
w^2=f_0 z^{n+1}+\cdots+f_nz \text{ if $n$ is odd,}$$
with $z\in \pi\sO$ under $\mu:z\to (1/z-\theta)=z(1-z\theta)$ modulo squares, for
non-Weierstrass points, except for $\infty^+$ and $\infty^-$.
\item[ad 10.] If $n$ is odd or $f_n$ is a square, then there are points $P\in C(k)$
with $x(P)=\infty$ and hence $z(P)=0$. We know that for such points, $\mu(P)=1$,
so we add that value to $W$.
\item[ad 11.] If $n$ is even and $f_n$ is a square, then $\infty^+$ and $\infty^-$
are rational points. However, the definition of $\tilde{\mu}$ does not yield the
correct value for these points, since $z(\infty^\pm)=0$. As a workaround, add
$x$ to $\tilde{G}$, so that the recursive search does not try to evaluate step 9
of \textsf{SquareClasses} for these points. The correct value has already been added to
$W$ in step 10.
\item[ad 12.] We now call \textsf{SquareClasses} to add to $W$ the images $\mu(P)$ of points
$P\in C(k)$ with $1/x(P)\in \pi\sO$. The nature of $\tilde{G}$ ensures that the
value passed to $c_0$ is irrelevant, so we pass a dummy value of $-1$.
Together with steps 3, 6, and 10, this
guarantees that after this, $W$ equals $\mu(C(k))$.
\end{enumerate}


\section{Computing the local image of $\mu$ at real places}\label{sec:real}

If $k$ is a completion of a number field at a complex place, then $A^*=A^{*2}$
for all $A=k[x]/f(x)$ with $f(x)$ a square-free polynomial. Furthermore $C(k)$
is non-empty for all curves $C$. In this case, $\mu(C(k)) = H_k = \{1\}$, so
there is nothing to do.

Now suppose that $k=\RR$ and that
$$f(x)=(x-\theta_1)\cdots(x-\theta_r)g(x)$$
where $\theta_1>\theta_2>\cdots>\theta_r$ are the real roots of $f(x)$
and $g(x)$ is a polynomial with no roots in $\RR$.

Then $A^*/A^{*2}=(\RR^*/\RR^{*2})^r\simeq(\ZZ/2\ZZ)^2$ and
$$\begin{array}{cccc}
\mu: &C(\RR)&\to&\RR^*/\RR^{*2}\times\cdots\times\RR^*/\RR^{*2}\\
&(x,y)&\mapsto&(x-\theta_1,\cdots,x-\theta_r)
\end{array}$$
Due to the ordering on the $\theta_i$, we have that if $x-\theta_i<0$ then
$x-\theta_j<0$ for $j\leq i$. For a point $P\in C(\RR)$ we have $f(x(P))\geq 0$,
so if $f_n>0$ then
$$\mu(C(\RR))=\{(1,\ldots,1),(-1,-1,1,\ldots,1),\ldots\},$$
which is to say, all vectors consisting of an even number of $-1$ entries followed by $1$ entries.
Conversely, if $f_n<0$ then
$$\mu(C(\RR))=\{(-1,1,\ldots,1),(-1,-1,-1,1,\ldots,1),\ldots\},$$
vectors consisting of an odd number of $-1$ entries followed by $1$ entries.

Note that, if $n$ is even, we have to quotient out by the subgroup generated by
$(-1,\ldots,-1)$, consisting of the image of $\RR^*$ in $A^*/A^{*2}$.

While the computation of $\mu(C(\RR))$ is quite straightforward, the use of this
information in computing $\selfake(C/k)$ for some number field $k$ is one of the
most error-prone parts due to precision issues. See Remark~\ref{rem:precision}
for more details.


\section{Computing the fake Selmer set}\label{sec:selfake}

In this section, let $k$ be a number field. We consider the algebra
$A=k[x]/(f(x))$. Let $S$ be the finite set of places $p$ satisfying one of:
\begin{itemize}
\item $p$ is infinite
\item $p$ has even residue characteristic
\item $f$ has coefficients that are not integral at $p$
\item the leading coefficient of $f$ is not a unit at $p$
\item $\ord_p(\disc(f))> 1$
\end{itemize}
We write $H_k(S)\subset H_k$ for the elements $\delta\in H_k$
such that $\rho_p(\delta)$ is unramified according to Definition~\ref{def:unram}
for all places $p\notin S$.
The first part of Lemma \ref{lemma:unram} asserts that $\selfake(C/k)\subset
H_k(S)$. It is a standard fact from algebraic number theory that the
subgroup $A(2,S) \subset A^*/A^{*2}$ of elements that are unramified outside~$S$
is finite, so $H_k(S)$ is a finite set.
This set can be computed, see the explanation of the \textsf{FakeSelmerSet}
algorithm below.

Let $T$ be the union of $S$ with the set of primes $p$ for which
$$\sqrt{q}+\frac{1}{\sqrt{q}} \le 2(2^{2g}(g-1)+1),\text{ where }q:=\#\sO_k/p\sO_k.$$
The second part of Lemma~\ref{lemma:unram} guarantees that for any prime
$p\notin T$ we will have
$\rho_p(H_k(S)) \subset \mu_p(C(k_p))$. Hence
$$\selfake(C/k)=\{\delta\in H_k(S): \rho_p(\delta)\in \mu_p(C(k_p)) \text{ for
all }p\in T\}.$$
This gives us a way to compute the fake Selmer-set explicitly.

\noindent\Cdef \textsf{FakeSelmerSet}$(f)$:

\begin{enumerate}
\item[1.] $A:=k[x]/(f(x))$
\item[2.] Let $S$ be the set of primes of $k$ described above.
\item[3.] \Cif $2\mid \deg(f)$:
\item[4.] \hspace{1em} $G:=A(2,S)/k(2,S)$
\item[5.] \Celse:
\item[6.]  \hspace{1em} $G:=A(2,S)$
\item[7.] $W:=\{g\in G: N_{A/k}(g)\in f_nk^{*2}\}$. \Cif $W=\emptyset$: \Creturn
$\emptyset$
\item[8.] $T:=S\cup \text{``small'' primes, as in Lemma~\ref{lemma:unram}}$
\item[9.] \Cfor $p\in T$:
\item[10.] \hspace{1em} $A_p:=A\otimes k_p$; $H_p':=A_p^*/A_p^{*2}$.
\item[11.] \hspace{1em} $W_p':=\mathsf{LocalImage}(f_p)\subset H_p'$ or, if $p\mid\infty$, use
Section~\ref{sec:real} to compute $W_p'$.
\item[12.] \hspace{1em} \Cif $2\mid \deg(f)$:
\item[13.] \hspace{2em} $H_p:=H_p'/k_p^*$; 
             $W_p:= \text{image of $W_p'$ in $H_p$}$
\item[14.] \hspace{1em} \Celse:
\item[15.] \hspace{2em} $H_p:=H_p'$; $W_p := W_p'$
\item[16.] \hspace{1em} Determine $\rho_p: G\to H_p$.
\item[17.] \hspace{1em} $W:=\{w\in W:\rho(w)\in W_p\}$.
\item[18.] \Creturn W
\end{enumerate}

\noindent\textbf{Explanation:}
\begin{enumerate}
\item[ad 3.] Following Remark~\ref{rmk:odddegree}, we also account for the
situation where $f$ is of odd degree.
\item[ad 4.] 
From Lemma~\ref{lemma:unram} it follows that $\selfake(C/k)$ can be represented by values in
$$A(2,S)=\{\delta\in A^*/A^{*2}:
 \delta\text{ is unramified in }(A\otimes k_p)^*/(A\otimes k_p)^{*2}
\text{ for all }p\notin S\}.$$
Let $S'\supset S$ contain generators for the $2$-parts of the class groups of
$k$ and the simple factors of $A$. We abuse notation slightly by writing
$A_{S'}^*$ for the $S'$-unit subgroup of $A^*$.  It is easy to verify that
$$A(2,S')=A_{S'}^*/A_{S'}^{*2}\text{ and that }A(2,S)\subset A(2,S').$$
Determining $A(2,S)\subset A_{S'}^*/A_{S'}^{*2}$ is a matter of $\FF_2$-linear
algebra.

In practice, determining $A_{S'}^*$ is the bottle-neck in these
computations, because it requires finding the class groups and unit groups of
the number fields constituting $A$.

We write $G$ for the classes representable by $A(2,S)$. It is
clear that $\selfake(C/k)\subset G$ and that $G$ is an explicitly
computable finite group.

\item[ad 7.] In all cases, the norm map induces a well-defined homomorphism
$G\to k^*/k^{*2}$, because $N_{A/k}(k^*)\subset k^{*2}$ if $2\mid \deg(f)$.

Furthermore, it may happen that $W$ is empty in this step. Since 
$\mu(C(k)) \subset \selfake(C/k) \subset W$, this implies that $C(k)$ is empty.

\item[ad 8.] As remarked in Lemma~\ref{lemma:unram}, we may obtain information
at primes of good reduction, if the size of the residue field is small. The
probability that these larger primes make a difference is rather small, and in
practice they often don't.

The theoretical size of $T$ grows extremely quickly: If $C$ is of genus $2$ then
$T$ should include all primes of norm up to $1153$ and for genus $3$ all primes
up to norm $66553$.

If it is infeasible to work with the full set $T$, one can work with a smaller
set of primes. The set we compute can then be strictly larger than $\selfake(C/k)$.

\item[ad 11.] In practice, $W$ will be a rather small set and the only reason we
want to compute $W_p$ is to reduce the size of $W$ in step 17. Especially for
large residue fields, $\mathsf{LocalImage}$ can be extremely expensive. By
integrating steps 11 through 17, one can detect early if $W_p$ is big enough to
cover all of $\rho(W)$. In that case, one does not have to compute the rest of
$W_p$ and can continue with the next $p$. This makes an immense difference in
running time in practice.

\item[ad 12.] Note that the implementation of $\mathsf{LocalImage}$ only
produces a set of representatives in $A_p^*/A_p^{*2}$ for $\mu_p(C(k_p))$. We
still have to quotient out by $k_p^*$ if $2\mid\deg(f)$.

\item[ad 16.] Since $G$ is a finite group, $\rho_p$ can simply be computed by
computing the images of the generators. However, one should take care that the
generators of $G$ in $A$ are represented by $S$-units. Algorithms naturally find
these with respect to a factor basis and writing the generators in another form
may be prohibitively expensive. One should instead compute the images of the
factor basis and take the appropriate linear combinations in an abstract
representation of the multiplicative groups.
\end{enumerate}

\begin{remark}\label{rem:precision} As is noted in Section~\ref{sec:real}, the
computation of $\mu_v(C(k_v))$ is quite straightforward for real places $v$.
The difficulty is in computing the map $\rho_v: H_k(S)\to H_{k_v}$. Any first
approach would probably involve representing $H_k(S)$ using generators of the
ring of $S$-units in $A_k$. Their images in real completions can lie very close
to $0$, making it necessary to compute very high precision approximations to
their real embeddings. As is remarked ad~16 above, a better approach is to
determine the signs of a factor basis and use the fact that $H_{k_v}$ lies in a
group to compute the images of $H_k(S)$.
\end{remark}


\section{Proving non-existence of rational points}\label{sec:noratpoints}

One of the most important applications of computing $\selfake(C/k)$ is that, if
it is empty, we can conclude that $C(k)$ is empty. This may even be the case if
$C$ does have points everywhere locally, and so it allows us to detect
failures of the local-to-global 
Principle.

\begin{example} Consider the hyperelliptic curve
$$C: y^2=2x^6+x+2.$$
Then $C(\QQ)$ is empty, but $C$ has points everywhere locally.
\end{example}

\begin{proof}
It is straightforward to check that $C$ does have points everywhere locally. In
this case, $A=\QQ[x]/(2x^6+x+2)$, which is a number field. Write $\sO$ for the
ring of integers in $A$. As it turns out, we have the prime ideal factorisation
$$2\sO=\fp\fq^5.$$
We have $\disc(2x^6+x+2)=2^4\cdot 11\cdot 271169$, so we know that
$\selfake(C/\QQ)\subset H_k(S)$ with $S=\{2,\infty\}$.

The ideal class group of $\sO$ is $\ZZ/2$ and $\fp$ and $\fq$ are not
principal ideals. Hence, there is no $S$-unit $u \in A$ such that
$N_{A/\QQ}(u) \in 2\QQ^{*2}$. Thus, in step 7 of \textsf{FakeSelmerSet}, we will find that $W$ is
empty and thus that $\selfake(C/\QQ)=\emptyset$ and therefore $C(\QQ)$ is empty.
\end{proof}
In this example, $H_{\QQ}(S)$ is empty, so there are no everywhere locally
solvable two-covers of~$C$. However, $H_{\QQ}$ is non-empty (the element
$2(\theta^5 - \theta^4 + \theta^2 - \theta + 1) \in A$ has norm~$18$, hence
represents an element of~$H_{\QQ}$, for example), so $C$ does have two-covers
of the form~$D_\delta$. This raises the following question, to which we do not
yet have an answer.

\begin{question} \label{Q}
Can a hyperelliptic curve over a number field $k$ be everywhere
locally solvable and yet have $H_k$ empty?
\end{question}

Another example that is worthwhile to illustrate, is that small primes of good
reduction can still yield information in the fake Selmer group calculation.

\begin{example}\label{ex:largegood} Consider the hyperelliptic curve
$$C: y^2=-x^{6} + 2x^{5} + 3x^{4} - x^{3} + x^{2} + x - 3$$
This curve has points everywhere locally over $\QQ$, has good reduction outside
$2$ and $35783887$, and has no rational points. One can show this by proving that
$\selfake(C/\QQ)$ is empty, but one needs to consider this curve locally at
$73$. In particular, this shows that $C$ has an unramified degree 16 cover over
$\QQ_{73}$, with good reduction and no $\FF_{73}$-points in the special fiber.
\end{example}

\begin{proof}
In this case the algebra $A$ is a number field with trivial class group. We
write $\sO$ for its ring of integers. We have
$$\disc(-x^{6} + 2x^{5} + 3x^{4} - x^{3} + x^{2} + x - 3)=2^2\cdot35783887$$
so in step 2 of \textsf{FakeSelmerSet}, we find that $S=\{2,\infty\}$. Since the class group
of $A$ is represented by prime ideals above $S$, we have
$A(2,S)=\sO_S^*/\sO_S^{*2}$, so generators are represented by a system of
fundamental units together with generators of the prime ideals above $2$:
$$\begin{array}{c|c|r}
&\alpha&N_{A/\QQ}(\alpha)\\
\hline
u_0&-1&1\\
u_1&\theta^5 - 2\theta^4 - 4\theta^3 + 2\theta^2 + 3\theta - 1&1\\
u_2&\theta^5 - 3\theta^4 + \theta^2 - 2\theta + 2&-1\\
u_3&\theta^4 - \theta^3 - 1&-1\\
p_2&\theta-1&-2\\
q_2&2/(\theta-1)^2&16\\
\hline
\end{array}$$
In Step 7 we find that $W$ is represented by $\{u_2,u_3,u_1u_2,u_1u_3\}$.

For $p=\infty$, the set $W$ does not get reduced. Note that $A$ has only two
real embeddings, corresponding to $\theta\mapsto 0.85$ and $\theta\mapsto 2.94$.
Since all representatives in $W$ have norm $-1$, we see that the real embeddings
of these elements will be of opposite sign, and since we are working modulo
$\QQ^*$, we can choose which is positive.
On the other hand, since the leading coefficient of $f$ is negative, the
algorithm in Section~\ref{sec:real} predicts that $\mu(C(\RR))=\{(-1,1)\}$,
which corresponds to the description above.

For $p=2$ the set $W$ does get reduced. We find that if
$(x,y)\in C(\QQ_2)$ then $x\in 2^2+O(2^3)$. It is only for
$\delta=u_1u_3=-\theta^5-\theta^4+1$ that we have that $\delta\in
(4-\theta+O(2^3))\QQ_2^*$ modulo squares in $(A\otimes\QQ_2)^*$, so if there is a
point $P_0\in C(\QQ)$ then $\mu(P_0)=u_1u_3$.

For $p=73$, we find in step 17 of
 \textsf{FakeSelmerSet} that $u_1u_3$ does not
map into the image of $C(\QQ_{73})$ and hence that $C(\QQ)$ is empty.
\end{proof}


\section{Applications to curves with points}\label{sec:ratpoints}

Let $k$ be a number field and let $C:y^2=f(x)$ be a curve of genus at least $2$
over $k$. Even if $C(k)$ is non-empty, the set $\selfake(C/k)$ still contains
useful information. If $\rank(\Jac_C(k))<\genus(C)$ and $\Jac_C(k)$ is actually
known, then one can use explicit versions of Chabauty's method
\cites{chabauty:original,coleman:effchab,flynn:flexchab} to compute a bound on
$\#C(k)$. In fact, if one combines this with Mordell-Weil sieving
\cites{bruelk:trinom, brusto:mwsieving, pooschst:F237}, then one would expect that one should be able to arrive
at a sharp bound \cite{poonen:heuristic}.

If $\rank(\Jac_C(k))\geq\genus(C)$ then one can try to pass to covers. One
chooses an unramified Galois cover $D/C$. By the Chevalley-Weil theorem
\cite{chevweil:covers}, the rational points of $C$ are covered by the rational
points of finitely many twists $D_\delta/C$ of $D/C$. For hyperelliptic curves
$C$, a popular choice is the $2$-cover $D_\delta$ described in
Section~\ref{sec:geom} \cites{bruintract, bruin:ellchab, brufly:towtwocovs}.
The genus of $D_\delta$ is much larger than the genus of $C$. This means that it
is possible that $\rank(\Jac_{D_\delta}(k))<\genus(D_\delta)$ for all relevant
$\delta$ and thus that Chabauty's method can be applied to each $D_\delta$.

The curve $D_\delta$ is usually of too high genus to do computations with
directly. However, over $\kbar$, the curve $D_\delta$ covers many hyperelliptic
curves besides $C$. These arise from factorisations $f(x)=g(x)h(x)$, where at
least one of $g,h$ has even degree. Suppose that $g(x)$ is monic and that its
field of definition is $L/k$, i.e., $g(x),h(x)\in L[x]$. We write
$$\begin{array}{rrcl}
E_\gamma:& \gamma y_1^2&=&g(x)\\
E'_\gamma:& (1/\gamma) y_2^2&=&h(x)\\
\end{array}$$
It is straightforward to see that, for every $\delta$, there is a value of
$\gamma=\gamma(\delta)\in L^*/L^{*2}$ such that, over $L$, we have the diagram
$$\xymatrix{
&D_\delta\ar[ddl]\ar[dd]\ar[ddr]\\
\\
E_\gamma\ar[dr]_x&C\ar[d]^x&E'_\gamma\ar[dl]^x\\
&\PP^1
}$$
Note that $D_\delta(k)$ maps to $C(k)$ and from there to $\PP^1(k)$. Therefore,
$D_\delta(k)$ maps to
$$\{P\in E_\gamma(L): x(P)\in \PP^1(k)\}$$
and in order to find which of those points correspond to points in $C(k)$ we
only have to find which points in $\PP^1(k)\cap x(E_\gamma(L))$ lift to $C(k)$.
There is ample literature on how to perform this last step
\cites{bruintract, bruin:ellchab, flynnweth:biell}. However, the problem
of finding the relevant values for $\gamma$ has largely been glanced over. This
is mainly because in any particular situation, it is quite easy to write down a
finite collection of candidates for $\gamma$ and then test, for every place $p$
of $k$ and any extension $q$ of $p$ to $L$, for each value if
$x(E_\gamma(L_q))\cap \PP^1(k_p)$ is non-empty. In fact, this is quite doable
for the fibre product $E_\gamma\times_{\PP^1} E'_\gamma$ too (see
\cite{bruintract}*{Appendix~A}). 

However, the smallest set of values for $\gamma$ we can hope to arrive at
through local means, is
$$\{\gamma(\delta): \delta\in\selfake(D/k)\}.$$
Also, note that the degree of $L$ over $k$ will usually be larger than the
degree of $A$. For instance, if $C:y^2=f(x)$ where $f(x)$
is a sextic with Galois group $S_6$ over $k$, the the field $L$ over which $f$
factors as a quadratic times a quartic is of degree $15$, while $A$ is of degree
$6$.
Hence, from a computational point of view it is interesting to
avoid as much computation as possible in $L$.

The map $\delta\mapsto\gamma(\delta)$ is in fact straightforward to compute,
given representatives $\delta\in A=k[x]/(f(x))$.
Let $L[\Theta]=L[x]/(g(x))$. Then there is a natural $k$-algebra homomorphism
$j:A\to L[\Theta]$ given by $\theta\mapsto\Theta$.
Using these definitions we have
$$\gamma(\delta)=N_{L[\Theta]/L}(j(\delta)).$$
While the degree of $L[\Theta]$ is probably quite high, we only need to compute
a norm with respect to it. This is not such an expensive operation.

The following example illustrates how the computation of the fake $2$-Selmer set
fits in with the standard methods for determining the rational points on
hyperelliptic curves.
\begin{example}
Let
$$C: y^2=2x^6+x^4+3x^2-2.$$
then $C(\QQ)=\{(\pm1,\pm2)\}$.
\end{example}

\begin{proof}
First we observe that $C$ covers two elliptic curves, $v_1^2=2u_1^3+u_1^2+3u_1-2$ and
$v_2^2=-2u_2^3+3u_2^2+u_2+2$, but each of these curves has infinitely many
rational points. When we apply \textsf{FakeSelmerSet} to this curve, we find
that $\selfake(C/\QQ)$ is represented by $\{-1-\theta,1-\theta\}$, so it is
equal to $\mu(C(\QQ))$. Putting $L=\QQ(\alpha)$, where $\alpha^2+\alpha+2=0$, we
obtain the factorization:
$$C: y^2=f(x)=(2x^2-1)(x^2-\alpha)(x^2+\alpha+1)$$
We choose $g(x)=(x^2-\frac{1}{2})(x^2-\alpha)$ and $h(x)=2(x^2+\alpha+1)$.
We find that $\gamma(1-\theta)=\gamma(-1-\theta)=\frac{1}{2}(1-\alpha)$ and
write $E:y_1^2=(1-\alpha)(2x^2-1)(x^2-\alpha)$.
Any point $(x,y)\in C(\QQ)$ must correspond to a point $(x,y_1)\in E(L)$ with
$x\in \QQ$. The curve $E$ is isomorphic over $L$ to the elliptic curve
$$\tilde{E}:v_3^2=u^3+(1-\alpha)u^2+(2-9\alpha)u+(16-2\alpha),$$
which has $\tilde{E}(L)\simeq (\ZZ/2)\times \ZZ$. This makes methods as
described in \cites{bruintract,flynnweth:biell} applicable and a $p$-adic
argument at $p=5$ proves that $x(E(L))\cap\PP^1(\QQ)=\{\pm 1\}$. 
\end{proof}


\section{Applications to genus $1$ curves}\label{sec:genus1}

In this section we illustrate how the computation of $\selfake(C/k)$ yields
interesting results even when $C$ is a genus $1$ double cover of $\PP^1$. We
recover well-known algorithms for doing $2$-descents and second $2$-descents
on elliptic curves. To our knowledge, nobody took the effort yet of implementing
second descent on elliptic curves over number fields, whereas our
implementation \cite{brusto:implementation} in MAGMA for computing
$\selfake(C/k)$ does work if $k$ is a general number field. We illustrate the
use by giving an example of an elliptic curve over $\QQ(\sqrt{2})$ with
non-trivial Tate-Shafarevich group.

If we have a genus $1$ curve of the form
$$E:y^2=f(x)=x^3+a_2x^2+a_4x+a_6,$$
then $\selfake(E/k)=\sel(E/k)$ is equal to the usual
$2$-Selmer group of $E$. In this case, the algorithm presented in
Section~\ref{sec:selfake} could be improved by using the fact that the sets computed in
\textsf{LocalImage} are groups of known size. One recovers an algorithm to
compute the $2$-Selmer group of an elliptic curve very similar to
\cite{cassels:lecec}.

Following Section~\ref{sec:geom}, for every $\delta\in\sel(E/k)$, we can write down an
everywhere locally solvable cover $D_\delta$. The model we obtain is the
intersection of two quadrics in $\PP^3$. The pencil of quadrics cutting out
$D_\delta$ contains a singular quadric $Q_0$ over $k$, however. Since $D_\delta(k_p)$
is non-empty for every $p$, the Hasse principle for conics mandates that $Q_0$
contains a line over $k$ and thus that the lines on $Q_0$ form a $\PP^1$ over
$k$. By sending a point on $D_\delta$ to the line on $Q_0$ that goes through
that point, we realise $D_\delta$ as a double cover of a $\PP^1$ and we obtain a
model of the form
$$C_\delta: Y^2=F_4 X^4+F_3 X^3+\cdots+F_0$$
where we know that $C_\delta$ is everywhere locally solvable.

Note that $C_\delta$ itself is again a curve of genus $1$. We can compute
$\selfake(C_\delta/k)$ in this case as well. The data obtained is the same as
that obtained by doing a second $2$-descent, along the lines of \cite{mss:seconddesc}. In this
case too, the algorithm could be optimised a bit by observing that
$\sel(C_\delta/k)$
maps surjectively to the fiber over $\delta$ with respect to
$\mathrm{Sel}^{(4)}(E/k)\to \sel(E/k)$; 
the fibers of the map $\sel(C_\delta/k) \to \mathrm{Sel}^{(4)}(E/k)$
are isomorphic to 
$E(k)[2]/2E(k)[4]$ (see section~6.1.3 in~\cite{stamminger:thesis}).
Thus $\sel(C_\delta/k)$ is
either empty or has a known cardinality. Similarly, in \textsf{LocalImage}, the
fact that the set $W$ carries a $\mu_{k_p}(E(k_p))$-action and is of known
cardinality can speed up the computation immensely.

\begin{example} Let $\alpha=\sqrt{2}$ and consider
$$E: y^2=x^3+(2-2\alpha)x+(2-9\alpha)$$
Then $\Sha(E/\QQ(\alpha))$ is non-trivial.
\end{example}

\begin{proof}
We find that $\delta=\theta^2+(8-4\alpha)\theta+13 - 6\alpha$ is in
$\selfake(E/\QQ(\alpha))$. The corresponding cover of $E$ can be given by the
model
\[ C : Y^2 = -(2 \alpha + 3) X^4 + (4 \alpha + 6) X^3 - (18 \alpha + 24) X^2
               + (16 \alpha + 24) X + 2 \alpha + 4 \,.
\]
It turns out that $\selfake(C/k)$ is empty. Hence, it follows that
$\delta\in\sel(E/k)$ is not in the image of $\mu(E(k))$ and therefore represents
a non-trivial member of $\Sha(E/\QQ(\alpha))[2]$. In fact, we have shown that
$\delta$ represents an element of $\Sha(E/\QQ(\alpha))$ that is not divisible by
$2$.
\end{proof}


\section{Efficiency of two-cover descent}\label{sec:statistics}

It is natural to ask how often we should expect that $\selfake(C/k)$ is empty if
$C(k)$ is. This is equivalent to determining what proportion of curves
$C$ over $k$ have an everywhere locally solvable two-cover 
(see Theorem~\ref{thm:twocovers}).
To quantify this question, let us limit ourselves to curves $C$ of genus $2$ and
$k=\QQ$. We define a na{\"\i}ve concept of height on the set of genus $2$ curves
over $\QQ$ so that we can define the proportion we are interested in as a limit.

\begin{definition} Let $M(D)$ be the set of genus $2$ curves over $\QQ$
of the form
$$y^2=f(x)$$
where $f(x)=f_6 x^6+\cdots+f_0\in\ZZ[x]$ and $|f_i|\leq D$ for $i=0,\ldots,6$
\end{definition}

Note that (especially when $D$ is large) an isomorphism class of a curve may be
represented by many different models in~$M(D)$.

It would be perhaps more natural to order the curves by $|\disc(f)|$ rather than maximal
absolute value of the coefficients. Our motivation for partitioning the set of
genus $2$ curves by $M(D)$ is that one can easily sample
uniformly from $M(D)$, whereas this is much more complicated otherwise.

A related question has a definite answer. In
\cites{poosto:locdens,poosto:casselstate} it is proved that there is a
well-defined proportion of genus $2$ curves over $\QQ$ that have points
everywhere locally. In fact, if one actually computes the local densities
involved, one arrives at
$$\lim_{D\to\infty}
\frac{\#\{C\in M(D): C(\QQ_p)\neq \emptyset \text{ for all }p\}}{\#M(D)}\approx
85\%$$
On the other hand, one would expect that the proportion of curves that
actually have a rational point vanishes as $D$ grows:
$$\lim_{D\to\infty}
\frac{\#\{C\in M(D): C(\QQ)\neq\emptyset\}}{\#M(D)}\stackrel{?}{=}0\,.$$
Heuristic considerations suggest that the quantity under the limit should
be of order $D^{-1/2}$. We are interested in the question whether the
following limit exists and what might be its value:
$$\lim_{D\to\infty}
\frac{\#\{C\in M(D): \selfake(C/\QQ)\neq \emptyset\}}{\#M(D)}.$$
To this end, we tested if $\selfake(C/\QQ)=\emptyset$ for a fairly large number
of curves, sampled from $M(D)$ for various $D$. For $D=1,2,3$ we considered all
of $M(D)$ (see \cite{brusto:smallgenus2}) and our results are unconditional for
all but 42 of the roughly 200\,000 isomorphism classes of curves involved.
In the table below, these curves are counted according to isomorphism class.

For $D=4,\ldots,60$ and for $D=100$ we have sampled curves $C$ from $M(D)$ uniformly randomly
and computed the following:
\begin{itemize}
\item Whether the curve has a small rational point (whose $x$-coordinate is a
rational number with numerator and denominator bounded by $10\,000$ in absolute
value),
\item Whether the curve is everywhere locally solvable,
\item Whether $\selfake(C/\QQ)=\emptyset$, where for reasons of efficiency, the class group information
needed to compute $A(2,S)$ was only verified subject to the generalized Riemann
hypothesis.
\end{itemize}
This places each curve in one of the four categories
\begin{itemize}
\item $C(\QQ_v)=\emptyset$ for some place $v$
\item $\selfake(C/\QQ)=\emptyset$ (but $C$ does have points everywhere locally)
\item $C(\QQ)$ contains a small point
\item It is unknown whether $C$ has a rational point or not.
\end{itemize}

\begin{table}[thb]
\begin{tabular}{|r|r|r|r|r|r|} \hline
&\multicolumn{2}{c|}{$\selfake(C/\QQ)=\emptyset$}&
\multicolumn{2}{c|}{$\selfake(C/\QQ)\neq\emptyset$} & \\
\hline
$D$&local obstruction&&
&C$(\QQ)\neq\emptyset$ & total \\
\hline
1&45&3&0&401 & 449 \\
2&2823&1096&29&14116 & 18064 \\
3&29403&27786&1492&137490 & 196171 \\\hline
10 & 5903 & 9915 & 1546 &20242 & 37606 \\
20 & 2020 & 4748 & 1012 & 5393 & 13173 \\
30 & 2717 & 6675 & 1579 & 5959 & 16930 \\
40 & 4025 &10648 & 2682 & 7963 & 25318 \\
50 &18727 &51831 &13538 &34269 &118365 \\
60 & 1589 & 4571 & 1278 & 2547 &  9985 \\
100& 8106 &24063 & 7045 &10786 & 50000 \\\hline
\end{tabular}

\medskip

\caption{Statistics on genus $2$ curves}\label{tbl:data}
\end{table}

\begin{figure}[htb]
\scalebox{0.5}{\includegraphics{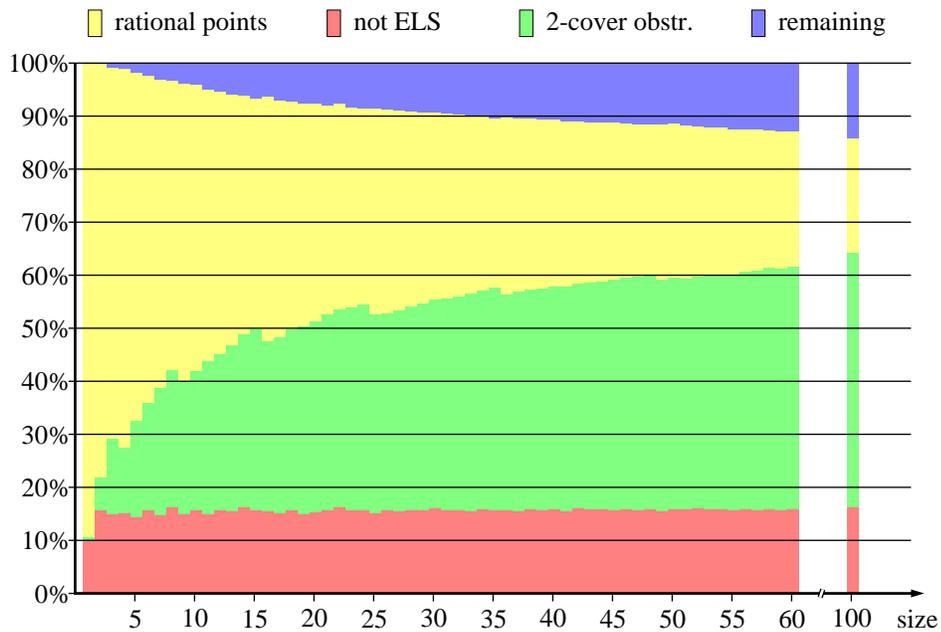}}
\caption{Proportion chart of obstructions}\label{fig:graph}
\end{figure}

See Table~\ref{tbl:data} for the statistics and Figure~\ref{fig:graph} for a graph
of the data.
It should be noted that the samples for the various~$D$
are not completely independent: Any curve
sampled from $M(D)$ that happened to have all its coefficients bounded by $D'<D$
was also included as sample from $M(D')$. 

We make a number of observations.
\begin{enumerate}
\item The proportion of curves with a local obstruction against rational
points tends to a value near $15$\% remarkably quickly.
\item As $D$ increases, the proportion of curves in $M(D)$ with a small
rational point decreases in the expected way. The jumps that can be observed
at $D = 4, 9, 16, \dots$ come from additional possibilities for points
at infinity or with $x = 0$ that occur when the leading or trailing coefficient
is a square.
\item The proportion of curves with $\selfake(C/\QQ)$ non-empty decreases 
more slowly. Figure~\ref{fig:graph} clearly shows that, at least for $C\in
M(100)$, testing if $\selfake(C/\QQ)=\emptyset$ is a very useful criterion to
decide if $C(\QQ)$ is empty, with less than 15\% of undecided curves.
\item The data is inconclusive on a possible limit value for the proportion of
curves $C\in M(D)$ with $\selfake(C/\QQ)=\emptyset$ as $D\to\infty$,
but it suggests that it might be somewhere between 65\% and~85\%.
It would be very interesting to find out if this limit exists and what
its approximate value might be. What makes this likely to be hard is the
subtle interplay between local and global information that determines
the size of~$\selfake(C/\QQ)$.
\end{enumerate}


\end{document}